\documentclass[11pt]{article}
\usepackage{amsmath}
\usepackage{ulem}
  \usepackage{graphics}
  \usepackage{amsfonts}
  \usepackage{epsfig}
  \usepackage{float}
  \usepackage{graphicx}
 \usepackage{epstopdf}
 \usepackage{amssymb}
 \usepackage{mathrsfs}
 \usepackage{amsthm}
 \usepackage[colorlinks=true]{hyperref}
 \usepackage[numbers]{natbib}
\usepackage{color}
\hypersetup{urlcolor=blue, citecolor=red}

\numberwithin{equation}{section}

\newtheorem{theorem}{Theorem}[section]
\newtheorem{corollary}{Corollary}[section]
\newtheorem{lemma}[theorem]{Lemma}

\newtheorem{proposition}{Proposition}[section]

\newtheorem{example}[theorem]{Example}

\newtheorem{definition}[theorem]{Definition}

\newtheorem{remark}[theorem]{Remark}

\DeclareMathOperator\MF{{\mathcal{MF}}}
\DeclareMathOperator\PMF{{\mathcal{PMF}}}
\DeclareMathOperator\M{Mod}
\DeclareMathOperator\T{Teich}
\DeclareMathOperator\K{{\mathcal H}}

\title{On a family of unitary representations of mapping class groups}
\author{Biao Ma}
\begin{document}
\date{}
\maketitle

\begin{abstract}
  For a compact surface $S = S_{g,n}$ with $3g + n \geq 4$, we introduce a family of unitary representations of the mapping class group $\M(S)$ based on the space of measured foliations. For this family of representations, we show that none of them \textcolor{black}{has almost} invariant vectors. \textcolor{black}{As one application}, we obtain an inequality concerning the action of $\M(S)$ on the Teichm\"{u}ller space of $S$. Moreover, using the same method plus recent results about weak equivalence, we also give a classification, up to weak equivalence, for the unitary quasi-representations with respect to geometrical subgroups.
\end{abstract}

\section{Introduction}
\bigskip
\noindent
Let $S = S_{g,n}$ be a compact, connected, orientable surface of genus $g$ with $n$ boundaries, the mapping class group $\M(S)$ of $S$ is defined to be the group of isotopy classes of orientation-preserving homeomorphisms of $S$ which preserving each boundary components (without the assumption that it should fix each boundary pointwise). Throughout this paper, $(g,n)$ is assumed to satisfy $3g + n \geq 4$ and a subsurface of $S$ is allowed to be disconnected.  \\

\noindent
Given a discrete group $G$, a unitary representation is a pair $(\pi, V)$ where $V$ is a Hilbert space and $\pi : G \rightarrow U(V)$ is a homomorphism from $G$ to the group of all unitary operators of $V$ \cite{BekkaHarpValette08}. Infinite dimensional unitary representations of mapping class groups $\M(S)$ received a lot of attention recently. In \cite{Paris2002}, the author considers unitary representations given by the action of $\M(S)$ on the curve complex associated to $S$. See \cite{Andersenvillemoes12}, \cite{AndersenVillemoes09},\cite{Goldman2005} for more topics in this direction.\\

\noindent
The group $\M(S)$ acts on the space of measured foliations \textcolor{black}{$\MF(S)$}, which is defined as the set of equivalence classes of measured foliations on $S$. As the action are ergodic with respect to generalized Thurston measures $\mu$ \cite{Masur1982},\cite{Masur1985},\cite{Lindenstrauss-Mirzakhani08}, \cite{Hamenstaedt2009} (see Section \ref{section:measure} for a brief description of the measures), one obtains a family of unitary representations by considering the induced action of $\M(S)$ on the space $L^2(\MF(S), \mu)$. It is quite easy to see that the family of unitary representations considered in \cite{Paris2002} is a special subfamily. However, unlike representations studied in \cite{Paris2002}, Example \ref{example:reducibility} will show that some of representations considered here are reducible. \\

\begin{definition}{\label{def}}
Let $(\pi, V)$ be a unitary representation of a discrete group $G$. The representation $\pi$ is said to \textcolor{black}{have almost} invariant vectors if for every finite set $K \subseteq G$ and every $\epsilon > 0$, there exists $v \in V$ such that $$ \max_{g \in K}\|\pi(g)v-v\| < \epsilon\|v\|.$$
\end{definition}
\medskip
\noindent
The main result of this paper is about the existence of almost invariant vectors for the representation $\pi^{\mu}$ associated to the action of $\M(S)$ on $L^2(\MF(S), \mu)$. The existence of such vectors for other representations of mapping class group has been discussed in \cite{Andersen2007b}.
\begin{theorem}[Theorem \ref{main:theorem}]
 For a compact surface $S = S_{g,n}$ with $3g + n \geq 4$ and each generalized Thurston measure $\mu$, the associated representation $\pi^{\mu}$ of $\M(S)$ does not have almost invariant vectors.
\end{theorem}

\noindent
The first direct application of this theorem is the following:
\begin{corollary}[Corollary \ref{main:coro2}]
 Let $S = S_{g,n}$ be a compact surface with $3g + n \geq 4$ and $\mu$ be a generalized Thurston measure, then $H^1(\M(S),\pi^{\mu}) = \overline{H^1}(\M(S),\pi^{\mu})$, where $\pi^{\mu}$ is the associated representation of $\M(S)$.
\end{corollary}

\noindent
For the second application, we will obtain a geometric inequality of independent interests concerning the action of $\M(S)$ on the Teichm\"{u}ller space $\T(S)$ of $S$.

\begin{corollary} [Corollary \ref{main:coro}]
Let $S = S_{g,n}$ be a compact surface with $3g + n \geq 4$ and $\gamma$ be the isotopy class of an essential simple closed curve on $S$. Then there exists a finite subset $\{\phi_1,...,\phi_n\}$ of $\M(S)$ consisting \textcolor{black}{of} pseudo-Anosov mapping classes and a constant $ \epsilon > 0 $, such that, for every point $\mathcal{X}$ in $\T(S)$, we have:
$$\max_{i\in \{1,2,...,n\}}\textcolor{black}{\left\{\sum_{\alpha \in \M(S).\gamma}e^{-2\ell_{\mathcal{X}}(\alpha)}(e^{\Delta^{\phi_i}_{\mathcal{X}}(\alpha)}-1)^2\right\}} \geq \epsilon\sum_{\alpha \in \M(S).\gamma} e^{-2\ell_{\mathcal{X}}(\alpha)},$$
where $\Delta^{\phi_i}_{\mathcal{X}}(\alpha) = \ell_{\mathcal{X}}(\alpha)-\ell_{{\phi_i.\mathcal {X}}}(\alpha)$ and $\ell_{\mathcal{X}}(\alpha)$ is the geodesic length of $\alpha$.
\end{corollary}

\noindent
For \textcolor{black}{unitary representations} associated to discrete measures on the space of measured foliations, some of them are irreducible and some are reducible. We will discuss irreducible decompositions (See Proposition \ref{proposition:irreducible}). We will also use the same method as \textcolor{black}{in the proof} of the main theorem, combined with recent results in \cite{BridsonHarpe},\cite{BKKO},\cite{BekkaKalantar}, to give a classification for a family of quasi-regular unitary representations, which is a stronger version of Corollary 5.5 in \cite{Paris2002}. Recall that, given two unitary representations $(\pi, \mathcal{H})$ and $(\phi, \mathcal{K})$ of a discrete group $G$, $\pi$ is {\it weakly contained} in $\phi$ if for every $\xi$ in $\mathcal{H}$, every finite subset $Q$ of $G$ and $\epsilon > 0$, there exist $\eta_1,...,\eta_n$ in $\mathcal{K}$ such that
$$\max_{g \in Q}\textcolor{black}{\left|<\pi(g)\xi, \xi> - \sum^{n}_{i=1}<\phi(g)\eta_i,\eta_i>\right|} < \epsilon.$$
If $\pi$ is weakly contained in $\phi$ and $\phi$ is weakly contained in $\pi$, then $\phi$ and $\pi$ are said to be {\it weakly equivalent}. \textcolor{black}{By Proposition F.1.7 in \cite{BekkaHarpValette08}, Definition \ref{def} is equivalent to say that the trivial representation is weakly contained in the representation $\pi$. }We then have the following theorem.
\begin{theorem}[Theorem \ref{theorem:classification}]
Let $S = S_{g,n}$ be a compact surface with $3g + n \geq 4$. Let $\gamma = \sum_{i=1}^k\gamma_i$ and $\delta=\sum_{i=1}^l\delta_i$, where $\{\gamma_i\}$ and $\{\delta_i\}$ are collections of pairwise disjoint, distinct isotopy classes of essential simple closed curves on $S$.
\begin{enumerate}
  \item If at least one of $k$ and $l$ is not $3g-3+n$, then the associated unitary representations $\pi_{\gamma}$ and $\pi_{\delta}$ are weakly equivalent if and only if $\gamma$ and $\delta$ are of the same topological type (that is, there is a mapping class $f$ so that $\gamma = f(\delta)$).
  \item Suppose $S$ is not $S_{0,4}, S_{1,1}, S_{1,2}, S_{2,0}$. If $k = 3g-3+n$, then $\pi_{\gamma}$ is weakly equivalent to the regular representation $\lambda_S$.
  \item Suppose $S$ is not $S_{0,4}, S_{1,1}, S_{1,2}, S_{2,0}$. If $k \neq 3g-3+n$, then $\pi_{\gamma}$ is not weakly contained in $\lambda_S$.
\end{enumerate}
\end{theorem}
\medskip
\noindent
This paper is organized as follows. Section \ref{section:2} is devoted to preliminary for group cohomology with coefficients in unitary representations. The proof of the main theorem is given in Section \ref{section:4}. \textcolor{black}{The proof is divided into two general lemmas: Lemma \ref{lemma:discrete} and Lemma \ref{lemma:nondiscrete}, and concluded by a technical statement, namely Proposition \ref{proposition:properly discontinuous}, concerning actions of subgroups of mapping class groups on $\MF(S)$}. Section \ref{section:3} is mainly devoted to this proposition and Section \ref{section:5} is for irreducible decompositions and the classification up to weak equivalence.

\subsection*{Acknowledgment}
This paper is originated in questions asked to the author by his  Ph.D. advisor Professor Indira Chatterji. He wish to thank her for many valuable discussions, help and useful comments. He is also grateful to Professor Fran\c{c}ois Labourie for providing a proof for Proposition \ref{prop: mcshine} and to Professors Vincent Delecroix, Ursula Hamenst\"adt, Yair Minsky and Alain Valette for discussions related to this work. Finally, the author would like to thank the referees for their valuable comments which helped to improve the manuscript. The author is supported by China Scholarship Council (NO. 201706140166).

\section{Cohomology with coefficients in representations}\label{section:2}
\bigskip
\noindent
{\bf Cohomology and reduced cohomology.}~~~
For a discrete group $G$ and a unitary representation $(V, \pi)$, one can talk about both cohomology and reduced cohomology group of G with coefficients in $\pi$. Definitions of cohomology and reduced cohomology of discrete groups with coefficients in a representation $\pi$ are  standard, so we refer to \cite{MartinValette07},\cite{Andersenvillemoes12},\cite{BekkaHarpValette08}. We briefly recall that one defines following vector spaces for a unitary representation $(V, \pi)$:
$$Z^1(G,\pi) \doteq \left\{b:G \rightarrow V |b(gh)=b(g) + \pi(g)b(h), ~\text{for all}~ g,h \in G \right\};$$
\begin{displaymath}
\begin{aligned}
B^1(G,\pi) \doteq & \{b \in Z^1(G,\pi) | \text{there exists}~ v \in V, ~\text{such that for all}~ g \in G,\\
&\phantom{=\;\;} b(g) = \pi(g)v - v \};
\end{aligned}
\end{displaymath}
$$H^1(G,\pi) \doteq Z^1(G,\pi) / B^1(G,\pi);$$
$$\overline{H^1}(G,\pi) \doteq Z^1(G,\pi) / \overline{B^1(G,\pi)},$$
where the closure in the last one is for uniform convergence. The vector space $H^1(G,\pi)$(resp. $\overline{H^1}(G,\pi)$) is the first (resp. reduced) cohomology group with coefficients in $\pi$.

\bigskip
\noindent
{\bf Almost invariant vectors.}~~~
 The following Guichardet's theorem \textcolor{black}{provides} a way to determine if $H^1(G) = \overline{H^1}(G)$.
\begin{theorem}[\cite{MartinValette07}]\label{almostinv:MV}
Let $G$ be a finitely generated discrete group and $(V,\pi)$ be a unitary representation without nonzero invariant vectors. Then the following two are equivalent:\\
1. The associated first reduced cohomology is the same as the first cohomology, that is, $H^1(G,\pi) = \overline{H^1}(G,\pi)$;\\
2. The representation $\pi$ does not have almost invariant vectors.
\end{theorem}

\noindent
One observation is that not having almost invariant vectors is closed under taking limit, more precisely, we have the following lemma.
\begin{lemma}\label{finiteness}
Let $(V,\pi)$ be a unitary representation of $G$ and $W$ be a $G$-invariant vector subspace of $V$ such that the closure $\overline{W} = V$. Then $\pi$ does not have almost invariant vectors if and only if the representation $\pi\mid_W$ in $W$ does not have almost invariant vectors.
\end{lemma}
\begin{proof}
Suppose that the pair $(K,\epsilon)$, where $K$ is a finite subset of G and $\epsilon > 0$, is given by the condition that $\pi|_W$ does not have almost invariant vector. Given any element $\xi \in V-W$, there is a sequence of elements $\{\xi_n\} \subseteq W$ such that $\xi_n \rightarrow \xi$ as $n \rightarrow \infty$. Then, for $n$ enough large, we have: $$\max_{g \in K}\parallel\pi(g)\xi-\xi\parallel = \max_{g \in K}\parallel\pi(g)\xi - \pi(g)\xi_n + \pi(g)\xi_n - \xi_n +\xi_n - \xi\parallel$$\noindent$$ \geq  \max_{g \in K}\parallel\pi(g)\xi_n -\xi_n \parallel -2\max_{g \in K}\parallel \xi_n - \xi\parallel \geq \epsilon \parallel \xi \parallel - \delta.$$
Now $\delta$ can be enough small, so $$\max_{g \in K}\parallel\pi(g)\xi-\xi\parallel \geq \epsilon \parallel \xi \parallel,$$ Which completes the proof of one direction. The opposite direction is obvious.
\end{proof}

\medskip
\noindent
Another easy observation is that, in order to show a representation of group does not have almost invariant vectors, one only need to pass to a subgroup. That is,
\begin{lemma}\label{almostinv:subgroup}
A unitary representation $(\pi, V)$ of a group $G$ does not have almost invariant vectors iff there exists a subgroup $H$ of $G$ such that the unitary representation $(\pi|_{H}, V)$ of $H$ does not have almost invariant vectors.
\end{lemma}

\bigskip
\noindent
{\bf Amenable groups.}~~~
A basic strategy in this article is to use the regular representation of free group $\mathbb{F}_2$ of rank 2, so the following theorem is of fundamental importance.
\begin{theorem}[\cite{Eymard72}]\label{Eym}
For the left regular representation $\pi$ of a finitely generated discrete group G on $\ell^2(G)$, $\pi$ \textcolor{black}{has almost} invariant vectors if and only if G is amenable.
\end{theorem}
\begin{remark}\label{Remark:Freegroup}
Since  $\mathbb{F}_2$ is not amenable, the left regular representation  of  $\mathbb{F}_2$ on $\ell^2( \mathbb{F}_2)$ does not have almost invariant vectors. We will regard $\ell^2( \mathbb{F}_2)$ as $\ell^2-$functions on vertices of the Cayley graph of  $\mathbb{F}_2$ with respect to a chosen generating set, and thus further identify $\ell^2( \mathbb{F}_2)$ with the vector space V, where $$V = \textcolor{black}{\left\{\sum_i\alpha_i g_i: \sum_i |\alpha_i|^2 < \infty, \alpha_i \in  \mathbb{C}, g_i \in \mathbb{F}_2\right\}}.$$
\end{remark}
\section{Generalized Thurston measures and dynamics on measured foliation spaces}\label{section:3}
\bigskip
\noindent
 In this section we will describe the integral theory on the space of measured foliations and the action of \textcolor{black}{subgroups of mapping class groups} on the space of measured foliations. A subgroup of $\M(S)$ in which all elements except the identity are pseudo-Anosov mapping classes will be called a pseudo-Anosov subgroup.

\subsection{Measures and $L^2-$theory on $\MF(S)$}

\textcolor{black}{\subsubsection{Generalized Thurston measures on $\MF(S)$}}\label{section:measure}

The space of measured foliations $\MF(S)$ of a surface $S$ is the set of equivalence classes of transversal measured (singular) foliations on $S$. Using train tracks, one can show that $\MF(S)$ has a piecewise linear integral structure such that $\M(S)$ acts on it as automorphisms (that is, preserves this piecewise linear integral structure)\cite{Thurston1979}. Therefore, in such local PL coordinates, $\M(S)$ acts as linear transformations.

\medskip
\noindent
A consequence of this PL structure is that $\MF(S)$ can be equipped with a $\M(S)-$invariant measure $\mu_{Th}$, called {\it the Thurston measure} on $\MF(S)$. Moreover, this measure can be generalized to obtain a family of locally finite, ergodic $\M(S)-$invariant measures $\mu^{[(\cal{R}, \gamma)]}_{Th}$ on $\MF(S)$ for complete pairs $(\mathcal{R}, \gamma)$, which will be called {\it generalized Thurston measures}. We present a brief summary of the construction of generalized Thurston measures $\mu^{[(\cal{R}, \gamma)]}_{Th}$ according to \cite{Lindenstrauss-Mirzakhani08}.

\medskip
\noindent
Let $\gamma = \sum_ic_i\gamma_i, c_i > 0$ be a multi-curve on $S$, that is, $\gamma$ is a collection of isotopy classes of  pairwise distinct, pairwise disjoint essential simple closed curves $\{\gamma_i\}$ on $S$ so that each curve has been weighted by $c_i > 0$. After fixing a hyperbolic structure on $S$, one can think a multi-curve $\gamma = \sum_ic_i\gamma_i, c_i > 0$ as a collection of simple closed geodesics $\{\widetilde{\gamma_i}\}$ on $S$ with $\widetilde{\gamma_i}$ labeled by a positive real number $c_i$, where $\widetilde{\gamma_i}$ is the unique geodesic representative in $\gamma_i$. We will use $\gamma$ to denote both the formal sum $\sum_ic_i\gamma_i$ and the subset $\bigsqcup\widetilde{\gamma_i}$ of $S$. Cutting $S$ along $\gamma$, one obtains a decomposition into a disjoint union $$\overline{S - \gamma} = \bigsqcup T_i,$$ where $\{T_i\}$ is a collection of subsurfaces of $S$ with boundary smoothly embedded in $S$. For $$\mathcal{R} = \bigsqcup S_i$$ with  $\{S_i\} \subseteq \{T_i\}$, the pair $(\mathcal{R}, \gamma)$ will be called {\it a complete pair}. For a complete pair $(\mathcal{R} = \bigsqcup S_i , \gamma)$, define $$\MF(\mathcal{R}) = \prod_{i} \MF^{*}(S_{i})$$ where $\MF^{*}(S_i) = \MF(S_i) \bigcup 0_{S_i}$ in which $0_{S_i}$ is the zero foliation on $S_i$. The space $\MF(\mathcal{R})$ can be $\M(\mathcal{R}, \gamma)-$embeded on $\MF(S)$ via enlarging boundary curves [See \cite{Fathi2012}, Expos\'{e} 6.6 for enlarging curves]. Denote by $\mathcal{M}(\mathcal{R})$ the image of this embedding. This set is endowed with the product measure $\mu_{\mathcal{R}} = \prod \mu^i_{Th}$, where $\mu^{i}_{Th}$ is the Thurston measure of $S_i$. Define also $$M(\mathcal{R}, \gamma) = \{\mathcal{F} + \gamma : \mathcal{F} \in \mathcal{M}(\mathcal{R})\} \subseteq \MF(S).$$ The inclusion induces a measure on $\MF(S)$, denoted by $\mu^{[(\cal{R}, \gamma)]}_{Th}$ and supported on the set of $\M(S)-$orbits of $M(\cal{R}, \gamma)$, from the product measure $\mu_{\mathcal{R}}$.

\medskip
\noindent
Special cases are when $\cal{R} = \emptyset$ and $\gamma$ is the isotopy class of a non-separating curve, or when $\cal{R} = S$ and $\gamma = \emptyset$. The corresponding measure in the case of $\cal{R} = \emptyset$ is a discrete measure, denoted by $\mu_{\gamma}$ and supported on $\M(S).\gamma$ which is regarded as a subset of $\MF(S)$, while in the case of $\gamma = \emptyset$ it is exactly the Thurston measure $\mu_{Th}$ on $\MF(S)$.

\medskip
\noindent
The following remarkable theorem indicates that generalized Thurston measures $\mu^{[(\cal{R}, \gamma)]}_{Th}$ are exactly all locally finite, $\M(S)-$invariant, ergodic measures on $\MF(S)$.
\begin{theorem}[Hamenst\"adt\cite{Hamenstaedt2009},Lindenstrauss-Mirzakhani\cite{Lindenstrauss-Mirzakhani08}]
Any locally finite $\M(S)-$invariant ergodic measure on $\MF(S)$, up to a constant multiple, is in the form of $\mu^{[(\cal{R}, \gamma)]}_{Th}$, where $(\mathcal{R}, \gamma)$ is a complete pair.
\end{theorem}

\subsubsection{Associated $L^2-$theory over $\MF(S)$}\label{section:des}
{\bf The case of discrete measures.}~~~
 Recall that when $\cal{R} = \emptyset$, $\mu^{[(\cal{R}, \gamma)]}_{Th}$ is the discrete measure supported on the set $\M(S).\gamma$, where $\M(S).\gamma$ is regarded as a subset of $\MF(S)$. We will first deal with the case that $\gamma$ is the isotopy class of an essential simple closed curve on $S$ and denote the measure by $\mu_{\gamma}$.

\medskip
\noindent
 Let $X_{\gamma} = \mathcal{C}^0_{\gamma}(S)$ be the subset of vertices of the curve complex consisting of $\M(S).\gamma$. By considering the Dirac measure supported on \textcolor{black}{$X_{\gamma}$}, one can define the Hilbert space $\ell^2(X_{\gamma})$. It is clear that  $\ell^2(X_{\gamma})$ is $\M(S)-$equivariantly isomorphic to $L^2(\MF(S),\mu_{\gamma})$. On the other hand, \textcolor{black}{let} $G_{\gamma} = \M(S,\gamma)$ \textcolor{black}{=$Stab_{\gamma}(\M(S))$} be the set of all elements in $\M(S)$ that fix $\gamma$, then $\ell^2(X_{\gamma})$ can be further $\M(S)-$equivariantly identified with $\ell^2(\M(S)/G_{\gamma})$. These two spaces give the same unitary representation of $\M(S)$, actually we have

\begin{theorem}[Paris\cite{Paris2002}]\label{cc:irreducible}
The infinite dimensional unitary representation of $\M(S)$ given by $\ell^2(\M(S)/G_{\gamma})$ is irreducible.
\end{theorem}
\begin{remark}\label{remark:multicurves}
This theorem was proved in a more general setting for \textcolor{black}{1-multi-curves on $S$, that is, $\gamma = \sum c_i\gamma_i$ with $c_i = 1$ for all $i$}.
\end{remark}
\noindent
Thus, in particular, this representation does not have non-zero invariant vectors. Meanwhile, the irreducibility also allows us to describe $\ell^2(\M(S)/G_{\gamma})$ more geometrically.

\bigskip
\noindent
The first description of $\ell^2(\M(S)/G_{\gamma})$ is classical. For $f\in \ell^2(X_{\gamma})$, let $Supp(f)=\{v \in X_{\gamma}: f(v) \neq 0 \}$. The function $f$ has {\it compactly support} if the cardinal of $Supp(f)$ is finite. Define the subspace $W$ of $\ell^2(X_{\gamma})$ as the set of elements in $\ell^2(X_{\gamma})$ which have compactly support. As $X_{\gamma}$ is discrete, the following notation will be used to represent $f \in W$: $f= \sum^n_{i=1}k_i\alpha_i$. Note that $W$ is $\M(S)-$invariant and the closure $\overline{W}$ of $W$ in $\ell^2(X_{\gamma})$ is then $\ell^2(X_{\gamma})$ itself. This description will be used in the proof of the main theorem in the case of discrete measures.

\bigskip
\noindent
The second description of $\ell^2(\M(S)/G_{\gamma})$ needs more explanations. Let $\T(S)$ be the Teichm\"{u}ller space of $S$, and for each point $\mathcal{X}$ of $\T(S)$, define a function on $X_{\gamma}$ by $$f_{\mathcal{X}}(\alpha)=e^{-\ell_{\mathcal{X}}(\alpha)}, \alpha \in X_{\gamma}$$ where $\ell_{\mathcal{X}}(\alpha)$ is the length of the unique geodesic in the isotopy class $\alpha$.
\begin{proposition}\label{prop: mcshine}
The function defined above is actually in $\ell^2(X_{\gamma})$.
\end{proposition}
\begin{proof}
 It amounts to say $$\sum_{\alpha \in X_{\gamma}}e^{-2\ell_{\mathcal{X}}(\alpha)} < \infty.$$ Thus this proposition is a corollary of the result of \cite{BirmanSeries85} or \cite{Mirzakhani08} about the polynomial growth of simple closed geodesics.
\end{proof}

\bigskip
\noindent
 Let $W'$ be the subspace of $\ell^2(X_{\gamma})$ which consisting of finite linear combinations of elements in $\{f_{\mathcal{X}}: \mathcal{X} \in \T(S)\}$. It is also easy to see that this subspace is $\M(S)-$invariant. Also by irreducibility, the closure $\overline{W'}$ of $W'$ is $\ell^2(X_{\gamma})$.
\begin{remark}
The second description gives rise to a parametrization for $\ell^2(X_{\gamma})$ via the Teichm\"{u}ller space, thus it can be viewed as a reply to Problem 2.5 in \cite{Goldman2005} for representations under consideration.
\end{remark}
\bigskip
\noindent
For the case of $\mathcal{R} = \emptyset$ and $\gamma$ is a general integral multicurve \textcolor{black}{$\gamma = \sum k_i\gamma_i$ with $k_i \in \mathbb{N}$}, Theorem \ref{cc:irreducible} is not true in general as shown by the following
\begin{example}\label{example:reducibility}
Consider the genus 2 closed surface $S$, regarded as a quotient along boundaries of holed sphere with four disjoint open disks deleted. Let $\gamma = 2\gamma_1 +3\gamma_2, \delta = \gamma_1 + \gamma_2$, where $\gamma_1$ and $\gamma_2$ are isotopy classes of two distinct images of boundaries. Obviously, there is a mapping class $s$ that permutes the $\gamma_i$'s. Denote $H = \M(S,\gamma)$ and $H' = \M(S,\delta)$, then we have the exact sequence: $$ 1 \rightarrow H  \rightarrow H' \rightarrow \mathbb{Z}_2 \rightarrow 1.$$ That is, $H$ is a normal subgroup of $H'$ of index 2. This exact sequence allows us to define a self-map of the left cosets $\{fH\}$ as follows. Write $H'$ as $H \bigsqcup sH$. There are two $\M(S)-$invariant bijections: $$\M(S).\gamma \leftrightarrow \{[g] = gH\},$$ $$\M(S).\delta \leftrightarrow \{[f] = fH'\}.$$ As $fH' = fH \bigsqcup fsH$, the set $\{gH\}$ can be rewritten as $\{fH,fsH\}$, this reformulation induces a well-defined inversion $i: fH = [f] \mapsto [fs] = fsH$.\\

\noindent A function $\phi$ on $G/H = \{gH\}$ is called even if for every $[g] \in G/H$, $\phi([g]) = \phi(i([g]))$ and a function $\varphi$ on $G/H$ is called odd if for every $[g] \in G/H$, $\varphi([g]) = - \phi(i([g]))$.\\

\noindent Define $V_1$ to be the subset of $\ell^2(G/H)$ consisting of even functions and $V_2$ to be the subset of $\ell^2(G/H)$ consisting of odd functions. It is easy to see that such two vector spaces are non-empty, closed and $\M(S)-$invariant subspaces of $\ell^2(G/H)$.
\end{example}
\begin{remark}\label{remark1}
For any discrete measure mentioned above, the associated unitary representation has no nonzero invariant vectors.
\end{remark}
\bigskip
\noindent
{\bf The case of non-discrete measures.}~~~
For general measures, we mention one remark.
\begin{remark}\label{remark2}
If $\mathcal{R}$ is nontrivial, ergodicity of the action shows that the associated unitary representation has no nonzero \textcolor{black}{invariant vectors}.
\end{remark}
\subsection{Actions of subgroups of $\M(S)$ on $\MF(S)$}
\medskip
\noindent
\textcolor{black}{{\bf Train tracks and a construction of pseudo-Anosov mapping classes.}}~~~ For later use, we first recall some facts about train tracks and a construction of pseudo-Anosov mapping classes by Thurston. All discussions here are standard and well-known, we refer to \cite{PennerHarer},\cite{Farb2012},[\cite{Fathi2012}, Expos\'{e} 13],\cite{Thurston1988} for more details.\\

\noindent A train track $\tau$ in a surface $S$ is an embedded smooth graph with extra conditions on vertices. A train track is called {\it recurrent} if it supports a positive transverse measure, that is, a measure assigns a positive number to every edge. A {\it transversely recurrent} train track is a train track such that every edge has a nontrivial essential transverse intersection with a simple closed curve. A {\it birecurrent} train track is thus a train track that both recurrent and transversely recurrent. A {\it maximal} birecurrent train track is a birecurrent train track that cannot be a proper subtrack of any other train track. Any measured foliation is carried by a maximal train track. We only remark here that, for a maximal birecurrent train track $\tau$, the set $E(\tau)$ of all positive transverse measures on $\tau$ is a positive linear submanifolds, that is, a subset of some Euclidean space defined by a family of linear equations with the condition that all parameters are positive. For the torus $T$, the set $\MF(T)$ of linear measured foliations can be covered by four affine charts $E(\tau_i)$ associated to four maximal birecurrent train tracks. We fix these four types of train tracks as blocks and denote them by $\{\tau_1,\tau_2,\tau_3,\tau_4\}$. See [\cite{PennerHarer}, Section 2.6, Figure 2.6.1] for such four train tracks in the annulus, thus in the torus. \\

\noindent We now sketch a construction of pseudo-Anosov mapping classes given by Thurston \cite{Thurston1988}. We only discuss Thurston's construction for closed surfaces. For surfaces with boundaries, one can modify the construction without any difficulty. Let $S = S_g (g \geq 2)$ and choose two essential simple closed curves $\alpha$ and $\beta$ on $S$ so that all connected components of $S-\alpha\bigcup\beta$ are open topological disks. For each intersection point $p$ of $\alpha$ and $\beta$, one can assign a rectangle to $p$ so that $S$ has a flat structure $\sigma$ and, with respect to this flat structure, both Dehn twists $T_{\alpha}$ and $T_{\beta}$ act as affine transformations (since we have flat structure, we can talk about affine transformations) with linear parts given by elements in $PSL(\mathbb{R})$. An element in the subgroup of $\M(S)$ generated by $T_{\alpha}$ and $T_{\beta}$ is pseudo-Anosov if it has a hyperbolic linear part.\\

\noindent We now mention some facts about the set $\mathfrak{L}(S,\sigma)$ of linear measured foliations on $S$ induced by the flat structure $\sigma$ above. Note that unstable and stable foliations of pseudo-Anosov mapping classes obtained by Thurstion's construction are in $\mathfrak{L}(S,\sigma)$ and $\mathfrak{L}(S,\sigma)$ is a closed subset of $\MF(S)$. If we arrange all rectangles mentioned above on the plane such that $\alpha-$sides are horizontal and label the rectangles from left to right by $\{\Box_1,\Box_2,...,\Box_m\}$, then a linear measured foliation $\mathfrak{F} \in \mathfrak{L}(S,\sigma)$ is given by parallel lines of the plane and a train track $\tau$ in $S$ carrying $\mathfrak{F}$ has the form that the restriction of $\tau$ in each rectangle $\Box_i$ is one of $\tau_i$ and all such $\tau_i$ appearing in $\tau$ are the same. Therefore there are four types of train tracks, denoted also by $\{\tau_1,\tau_2,\tau_3,\tau_4\}$, so that $\mathfrak{L}(S,\sigma) \subseteq \bigcup_{i=1}^{4}E(\tau_i)$. A direct computation shows that linear measured foliations on $S$ induced by this flat structure are determined by weights on two edges of $\tau_i \bigcap \Box_1$, thus each $\mathfrak{L}(S,\sigma) \bigcap E(\tau_i)$ is parameterized by two free independent parameters.

\begin{lemma}\label{lemma:null}
Let $S = S_{g,n}$ be a compact surface with $3g+n \geq 5$, then each $\tau_i$ is birecurrent and the set $\mathfrak{L}(S,\sigma)$ of linear measured foliations with respect to a flat structure $\sigma$ constructed as described above is of null $\mu_{Th}-$measure.
\end{lemma}
\begin{proof}
It is obvious that each $\tau_i$ is birecurrent. We divide the proof of the rest into two cases according to whether $\tau_i$ is maximal or not. If $\tau_i$ is not maximal, then any measured foliation carried by $\tau_i$ is not maximal \cite{PennerHarer}. By [\cite{Lindenstrauss-Mirzakhani08}, Lemma 2.3], $E(\tau_i)$ has null $\mu_{Th}-$measure. If $\tau_i$ is maximal, then, as $\tau_i$ is a birecurrent train track, $E(\tau_i)$ is an open subset of $\MF(S)$ and thus every point in $E(\tau_i)$ should be determined by weights on $6g-6+2n$ edges of $\tau_i$. As remarked above that $E(\tau_i) \bigcap \mathfrak{L}(S,\sigma)$ is determined by weights on two edges of $\tau_i \bigcap \Box_1$ which can be extended to obtain $6g-6+2n$ free parameters of $E(\tau_i)$. That is to say, $E(\tau_i) \bigcap \mathfrak{L}(S,\sigma)$ is locally given by $x_3=x_4=...=x_{6g-6+2n}=0$ in $\mathbb{R}^{6g-6+2n}$ whose coordinates is given by $\{x_1,...,x_{6g-6+2n}\}$. Therefore, $E(\tau_i) \bigcap \mathfrak{L}(S,\sigma)$ is a null set. Since $\mathfrak{L}(S,\sigma) \subseteq \bigcup_{i=1}^{4} E(\tau_i)$, hence $\mathfrak{L}(S,\sigma)$ is a null set as well.
\end{proof}
\medskip
\noindent
{\bf Almost properly discontinuous action.}~~~
We introduce a concept for a group action on a Borel space (that is, a topological space endowed with a Radon measure) which is weaker than usual properly discontinuous action.
\begin{definition}
Let G be a group and $(X,\mu)$ be a Borel space. Suppose that G acts on X by measure-preserving homeomorphisms. We say that G acts on X almost properly discontinuously if there exists a  G-invariant subset $K$ with $\mu(K) = 0$ such that G acts on $X - K$ properly discontinuously.
\end{definition}
\begin{example}\label{example:Schottky}
Let $H \leq \textcolor{black}{PSL(2,\mathbb{Z})}$ be a Schottky group, then its limit set $\Lambda(H) \subseteq S^1$ \textcolor{black}{, as a Cantor set}, has zero Lebesgue measure, and \textcolor{black}{thus it acts on $\{\mathbb{R}^2-(0,0)\}/\{\pm1\}$ almost properly discontinuously}.
\end{example}
\noindent
Although the action of $\M(S)$ on $\MF(S)$ are ergodic with respect to generalized Thurston measures, the action of subgroups of $\M(S)$ on $\MF(S)$ is not always ergodic. The following proposition allows us to use properties of the ~``properly discontinuous" action.
\begin{proposition}\label{proposition:properly discontinuous}
For each complete pair $(\cal{R},\gamma)$, there exists a rank 2 free pseudo-Anosov subgroup $H$ of $\M(S)$ that acts on $\MF(S)$ almost properly discontinuously with respect to the generalized Thurston measure $\mu^{[(\cal{R}, \gamma)]}_{Th}$.
\end{proposition}
\noindent
Any such free group will be called  {\it a $p$-rank 2 free subgroup}.\\

\noindent The first case is when $\cal{R} = \emptyset$ or each component of $\cal{R}$ is $S_{0,3}$, then this proposition is obvious by taking $H$ to be any free pseudo-Anosov subgroup generated by two pseudo-Anosov mapping classes (this works the same for non-integral multicurves as for integral multicurves). For other cases, we prove this proposition through two lemmas.

\begin{lemma}\label{lemma:p-rank}
There exists a \textcolor{black}{$p$-rank 2 free} subgroup $H$ of $\M(S)$ that acts on $\MF(S)$ almost properly discontinuously with respect to the Thurston measure $\mu_{Th}$.
\end{lemma}
\begin{proof}
If $S = S_{0,4}~~or~~S_{1,1}$, then, in both cases, $\MF(S)$ can be identified with \textcolor{black}{$\{\mathbb{R}^2-{(0,0)}\}/\{\pm1\}$} and $\PMF(S)$ can be identified with $S^1$.  Moreover, there is a finite index subgroup of $\M(S)$ such that the action of this subgroup on $\PMF(S)$ is equivalent to the action of $PSL(2,\mathbb{Z})$ on $S^1$, see [\cite{Farb2012},Chapter 15] for the case of $S_{0,4}$. By taking $H$ to be any subgroup given in Example \ref{example:Schottky} and considering the set $Y = Pr^{-1}(\Lambda(H))$, where $Pr: \MF(S) \rightarrow \PMF(S)$ is the projection, the action of $H$ on $\MF(S)$ is thus almost properly discontinuous and $\mu_{Th}(Y) = 0$.\\
For other $S$, we deduce this lemma by first passing to $\PMF(S)$ and then using the result of \cite{McCarthy-Papadou89} on limit sets. \textcolor{black}{Let $\phi$ and $\psi$ be two independent pseudo-Anosov mapping classes obtained by Thurston's constrcution. By the ping-pong lemma, one can construct a free pseudo-Anosov subgroup $H$ generated by some powers of $\phi$ and $\psi$. As remarked before that stable and unstable measured foliations of pseudo-Anosov elements in $H$ are linear measured foliations and $\mathfrak{L}(S,\sigma)$ is a closed subset, therefore, by Lemma \ref{lemma:null}, the limit set $\Lambda(H)$ of $H$, which is defined to be the closure of the set of fixed points of non-trivial elements of $H$ with respect to the action on $\PMF(S)$, has the property that $$\mu_{TH}(Pr^{-1}(\Lambda(H))) = 0.$$ On the other hand, one can define the zero set $Z(\Lambda(H)) (\subseteq \PMF(S))$ of $\Lambda(H)$ \cite{McCarthy-Papadou89}. By combining with facts [See \cite{McCarthy-Papadou89}, Proposition 6.1] that $Z(\Lambda(H)) - \Lambda(H)$ consists of no uniquely ergodic foliations and uniquely ergodic foliation has full $\mu_{Th}-$measure, we know that $Pr^{-1}(Z(\Lambda(H)))$ has null $\mu_{Th}-$measure. By [\cite{McCarthy-Papadou89},Theorem 7.17], $H$ acts properly discontinuously on $\PMF(S)- Z(\Lambda(H))$, thus properly discontinuously on $\MF(S)-Pr^{-1}(Z(\Lambda(H)))$. Hence $H$ acts almost properly discontinuously on $\MF(S)$.}
\end{proof}

\noindent For $\cal{R} \neq S$, a complete pair  $(\cal{R}, \gamma)$ is called \textit{a middle type} if $\cal{R} \neq \emptyset $ and there is a connected component $\neq S_{0,3}$.
\begin{lemma}
For a complete pair $(R, \gamma)$ of middle type, there exists a \textcolor{black}{$p$-rank 2 free} subgroup $H$ of $\M(S)$ that acts on $\MF(S)$ almost properly discontinuously with respect to the measure $\mu^{[(R, \gamma)]}_{Th}$ .
\end{lemma}
\begin{proof}
 We will follow the idea of [\cite{Lindenstrauss-Mirzakhani08}, Lemma 3.1] to prove this lemma. Fix any hyperbolic structure $X$ on $S$ and consider the continuous function $\ell_{X} : \MF(S) \rightarrow R_{+}$ extending the geodesic length function. Thus $$\MF(S) = \lim_{L_1 \rightarrow 0, L_2 \rightarrow \infty }B^{L_1}_{L_2}(X),$$ where $B^{L_1}_{L_2}(X) = \{\nu \in \MF(S): \ell_{X}(\nu) \in [L_1,L_2]\}$ is a compact set and, as pointed out in the proof of [\cite{Lindenstrauss-Mirzakhani08}, Lemma 3.1], $B^{L_1}_{L_2}(X) \bigcap (\bigcup_{g \in \M(S)} g.M(\cal{R}, \gamma))$ is equal to $B^{L_1}_{L_2}(X) \bigcap (\bigcup _ {i = 1}^{ n} g_{i}.M(\cal{R}, \gamma)),$ for some finite set $\{g_1, ..., g_n\} \subset \M(S)$. Fix a free pseudo-Anosov subgroup $H$ of $\M(S)$ and take any compact subset $K \subseteq \bigcup _{g \in \M(S)} g.M(\mathcal{R}, \gamma)$. Taking $L_1$ small enough and $L_2$ large enough, one can assume $K \subseteq B^{L_1}_{L_2}(X)$. We now claim that $$|\{ h \in H: h.K \bigcap K \neq \emptyset\}| < \infty.$$ Let $Z = \M(S).\gamma$ and $\ell_{X}: Z \rightarrow R_{+}$. We first claim that there is a finite set $J \subseteq Z$ such that $$\{h \in H: h.K \bigcap K \neq \emptyset\} \subseteq \{h \in H: h.J \bigcap J \neq \emptyset\}.$$ For every element in $K$ can be written as $\gamma + \nu$ such that $\ell_{X}(\gamma)$ is bounded. If $h.K \bigcap K \neq \emptyset$, then $h(\gamma)$ also has bounded $\ell_{X}-$length and all bounds can be chosen to be uniform on $K$, say $[a,b]$. Since $\ell_{X}$ is a proper map on $Z$ (that is, the inverse of compact set is also compact), $J = \ell_{X}^{-1}([a,b])$ is then a finite subset of $Z$ containing both $h(\gamma)$ and $\gamma$. So one has $\{h \in H: h.K \bigcap K \neq \emptyset\} \subseteq \{h \in H: h.J \bigcap J \neq \emptyset\}$. By the discussion of the case $\cal{R} = \emptyset$, the set $\{h \in H: h.J \bigcap J \neq \emptyset\}$ is finite which implies that the finiteness of $|\{ h \in H: h.K \bigcap K \neq \emptyset\}|$. Now taking the measure zero set to be $Y = \MF(S) - \bigcup_{g \in \M(S)} g.M(\cal{R}, \gamma)$ completes the proof.
\end{proof}

\bigskip
\noindent
{\bf $H-$related cover.}~~~
\textcolor{black}{Given a group $H$ and a Borel space $(X,\mu)$. Suppose that $H$ acts on $X$ almost properly discontinuously and freely. Examples for such $(H,X,\mu)$ are given by Proposition \ref{proposition:properly discontinuous}. By definition of almost properly discontinuous action, there is a null set $Y$ such that $H$ acts on $X-Y$ properly discontinuously. For any compact subset $K$ of $X - Y$}, we will describe a ~``nice" cover of $K$. Since $X - Y$ is the domain of discontinuity of $H$, for every $p$ in $K$, there is an open neighbourhood $\mathcal{U}_p$ of $p$ in $X-Y$ with finite $\mu-$measure such that for all $h \in H$, one has $ h.\mathcal{U}_p \bigcap \mathcal{U}_p = \emptyset$. Thus there is an open cover of $K$. By compactness of $K$, choose a finite sub-cover of this cover. Label the sub-cover by $\mathcal{U}_1,...,\mathcal{U}_n$ and for each $i \in {1,...,n}$, consider $ A_i = \{h.\mathcal{U}_i| h \in H\}$. Starting from $i=1$, form a family $B_1 = \{X_k \in A_1| X_k \bigcap K \neq \emptyset\}$ as well as $C_1 = \{Y_k|Y_k=X_k \bigcap K,X_k \in B_1\}$. Delete $\bigcup_{Y_k \in C_1} X_k$ from $K$ and denote the resulting compact set by $K_1$. Then for $K_1$, there is a family $B_2 = \{X_k \in A_2| X_k \bigcap K_1 \neq \emptyset\}$ as well as $C_2 = \{Y_k|Y_k=X_k \bigcap K_1,X_k \in B_2\}$. Delete $\bigcup_{Y_k \in C_2} X_k$ from $K_2$ and denote the resulting compact set by $K_3$. Continuing this process, there is a cover of K which can be written in the following formula: $$K \subseteq \bigsqcup_{k=1}^n\bigsqcup_{Y_i \in C_k}Y_i.$$ So $K$ can be covered by finite many pairwise disjoint $\mu-$measurable sets \textcolor{black}{(we allow some of them to be null sets)}. This will be called  an {\textit{$H-$related cover of K}}, since, for each $k$, $C_k$ is a family of disjoint sets that lie inside the $H-$orbit of some set.

\section{Nonexistence of almost invariant vectors}\label{section:4}

\bigskip
\noindent
Let $\K(\mu) = L^2(\MF(S), \mu)$, where $\mu = \mu^{[(\cal{R}, \gamma)]}_{Th}$ is a generalized Thurston measure explained in Section 3.1.1, and $\pi^{\mu}$ be the associated unitary representation of $\M(S)$. The main result of this section is the following:
\begin{theorem}\label{main:theorem}
 For a compact surface $S = S_{g,n}$ with $3g + n \geq 4$ and each generalized Thurston measure $\mu$, the associated representation $\pi^{\mu}$ of $\M(S)$ does not have almost invariant vectors.
\end{theorem}

\noindent
By using Theorem \ref{almostinv:MV}, Remark \ref{remark1} and Remark \ref{remark2}, we have:
\begin{corollary}\label{main:coro2}
 Let $S = S_{g,n}$ be a compact surface with $3g + n \geq 4$ and $\mu$ be a generalized Thurston measure, then $H^1(\M(S),\pi^{\mu}) = \overline{H^1}(\M(S),\pi^{\mu})$, where $\pi^{\mu}$ is the associated representation of $\M(S)$.
\end{corollary}
\begin{proof}
By Theorem \ref{almostinv:MV}, we only need to show that the representation $\pi^{\mu}$ has no nonzero invariant vectors. The corollary is thus concluded by using Remark \ref{remark1} for discrete measures and Remark \ref{remark2} for non-discrete measures.
\end{proof}
\bigskip
\noindent
Let $\gamma$ be the isotopy class of an essential simple closed curve on $S$, $X = \M(S).\gamma$ and $\mathcal{X}$ be a point in the Teichm\"{u}ller space $\T(S)$ of $S$. Denoting $\Delta^{\phi_i}_{\mathcal{X}}(\alpha) = \ell_{\mathcal{X}}(\alpha)-\ell_{\phi_i.\mathcal{X}}(\alpha) $, where $\alpha \in X$, and using the description of $\ell^2(X)$ via $\T(S)$ in Section \ref{section:des}, the following inequality is easy to show:
\begin{corollary}\label{main:coro}
Let $S = S_{g,n}$ be a compact surface with $3g + n \geq 4$ and $\gamma$ be the isotopy class of an essential simple closed curve on $S$. Then there exists a finite subset $\{\phi_1,...,\phi_n\}$ of $\M(S)$ consisting of pseudo-Anosov mapping classes and a constant $ \epsilon > 0 $, such that, for every point $\mathcal{X}$ in $\T(S)$, we have:
$$\max_{i \in \{1,2,...,n\}}\textcolor{black}{\left\{\sum_{\alpha \in \M(S).\gamma}e^{-2\ell_{\mathcal{X}}(\alpha)}(e^{\Delta^{\phi_i}_{\mathcal{X}}(\alpha)}-1)^2\right\}} \geq \epsilon \sum_{\alpha \in \M(S).\gamma} e^{-2\ell_{\mathcal{X}}(\alpha)}.$$
\end{corollary}

\bigskip
\noindent
\textcolor{black}{We divide the proof of Theorem \ref{main:theorem} into two lemmas}. First we prove a lemma used for discrete measures.
\begin{lemma}\label{lemma:discrete}
Let $G$ be a discrete countable group and $X$ be a discrete set equipped with a $G-$action. Suppose that there is a rank 2 free subgroup $H$ of $G$ such that $H$ acts on $X$ freely. Then the unitary representation $\pi = \ell^2(X)$ of $G$ associated to the action of $G$ on $X$ does not have almost invariant vectors.
\end{lemma}
\begin{remark}
This lemma is well-known, we give an elementary proof here mainly for heuristic purposes.
\end{remark}
\begin{definition}
Let $H$ be a rank 2 free group and $X$ be a space that $H$ acts. Suppose $x \in X$ such that the stabilizer $Stab_{H}(x)$ of $x$ is trivial. The image of $H$ under the orbit map $H \rightarrow X, h \mapsto h.x$ is called the 2-tree based at $x$ (with respect to $(H, X)$).
\end{definition}
\begin{proof}[Proof of Lemma \ref{lemma:discrete}]
By Lemma \ref{almostinv:subgroup}, we can pass to subgroups. For the action of the group $H$ on the space $X$ and any point $p \in X $, consider the 2-tree based at $p$ with respect to $(H, X)$.\\
\noindent
Let $W$ be the subspace of $\ell^2(X)$ consisting of functions with finite support. As $W$ is $G-$invariant and dense, by Lemma \ref{finiteness}, it is enough to show that $(\pi|_{W}, W)$ does not have almost invariant vectors. That is, we have to find $(K, \epsilon)$ with the property that $$\max_{g \in K}\|\pi(g)f-f\|^2 \geq \epsilon\|f\|^2,~\text{for all}~f\in W.$$  Since $H \cong \mathbb{F}_2$, as mentioned in Remark \ref{Remark:Freegroup}, the left regular representation $\ell^2(H)$ does not have almost invariant vectors, thus such a pair $(K, \epsilon)$ exists for the regular representation. Fix such pair $(K, \epsilon)$ for the rest of the proof. Here are two facts.\\
 \noindent
 \underline{\textbf{Facts:}}\\
 \noindent
 1. For every 2-tree $\mathbb{T}$ based at a point, $\ell^2(\mathbb{T}) $ is $H-$equivariantly isomorphic to $\ell^2(H)$.\\
 2. Different 2-trees are disjoint and thus, if the support $A_1$ of $f_1 \in \ell^2(X)$ and the support $A_2$ of $f_2 \in \ell^2(X)$ are located in different 2-trees, then $f_1$ and  $f_2$ are orthogonal.

 \medskip
 \noindent
These two facts imply that we only need to deal with $\ell^2-$functions on $X$ whose finite support contained in a single 2-tree. In fact, for every $f \in W$, if we decompose its support $K_f$ as $$K_f = \bigsqcup^{n}_{i=1} K_{f_i},$$ where $K_{f_i}$ lie in different 2-trees and $f_i$ is defined to be the restriction of $f$ on such different 2-trees, then $$f = \sum^{n}_{i = 1}f_i,$$ $$ \|\pi(g)f-f\|^2 = \sum^{n}_{i = 1} \|\pi(g)f_i-f_i\|^2,~\text{for all}~g \in K.$$ Note that $K \subseteq H$ is fixed. If the support of $f_i$ is contained in a 2-tree $\mathbb{T}_i$, by Remark \ref{Remark:Freegroup}, there exists $g_i \in K$ such that$$ \|\pi(g_i)f_i-f_i\|^2 \geq \epsilon\|f_i\|^2.$$ Now for every $f_i$, let $g_i$ be an element satisfying the above inequality. If two 2-trees $f_i,f_j$ correspond to the same $g_i = g_j$, then $f_i + f_j$ also satisfies that inequality. As $K$ is finite, denote $\sharp{K} = m$ and so $f$ can be further decomposed, that is, $f = f'_1 + f'_2 + \cdots +f'_s (s \leq m)$ such that $f'_k = \sum_{j} f_{jk},$ where $f_{jk} \in \{f_1,...,f_n\}$ and $\{f_{jk}\}_j$ correspond to the same $g_k \in K$. We claim that there exists $g_l \in K$ such that $$ \|\pi(g_l)f-f\|^2 \geq \frac{\epsilon}{s}\|f\|^2 \geq \frac{\epsilon}{m}\|f\|^2.$$ Otherwise, since for all $g_i$ selected, we have
 \begin{equation}\label{proof:equation}
 \|\pi(g_i)f-f\|^2 \geq \|\pi(g_i)f_i-f_i\|^2 \geq \epsilon\|f_i\|^2,
 \end{equation}
 then
 \begin{displaymath}
 \begin{aligned}
  &\epsilon\|f\|^2 = \sum_{i=1}^{m}\frac{\epsilon}{m}\|f\|^2 > \sum_{i=1}^{m}\|\pi(g_i)f-f\|^2 \\
  &\phantom{=\;\;}\geq \sum_{i= 1}^{s}\|\pi(g_i)f-f\|^2 \geq \sum_{i= 1}^{s}\epsilon\|f_i\|^2 =  \epsilon\|f\|^2.
 \end{aligned}
 \end{displaymath}
 The second inequality is the assumption and the last inequality is inequality (\ref{proof:equation}). Thus there exists a pair $(K, \eta = \frac{\epsilon}{\sharp K})$ such that $$\max_{g \in K}\|\pi(g)f-f\|^2 \geq \eta \|f\|^2,~\text{for all}~f \in W.$$ So the proof of  the lemma is completed.
\end{proof}
\bigskip
\noindent
Then we prove a lemma used for non-discrete measures.
\begin{lemma}\label{lemma:nondiscrete}
Let $G$ be a discrete countable group and $(X,\mu)$ be a Borel space. Suppose that $G$ acts on $X$ by measure-preserving homeomorphisms. If there exists a rank 2 free subgroup $H$ of $G$ such that $H$ acts on $X$ almost properly discontinuously and freely, then the unitary representation $\pi = L^2(X,\mu)$ of $G$ associated to the action of $G$ on $X$ does not have almost invariant vectors.
\end{lemma}
\begin{proof}[Proof of Lemma \ref{lemma:nondiscrete}]
 Also by Lemma \ref{almostinv:subgroup}, we can pass to subgroups. Fix a null subset $Y$ of $X$ such that $H$ acts on $X-Y$ properly discontinuously. For any point $p \in X $, consider the image of $H$ under the orbit map, given by $$h \longmapsto h.p.$$ Since the stabilizer $Stab_{p}(H)$ is trivial, this map is injective. This is the 2-tree based at $p$ with respect to $(H, X)$. Define $W$ to be the $G-$invariant subspace of $L^2(X,\mu)$ consisting functions $f \in L^2(X,\mu)$ that compactly supported on $X - Y$. Thus $\overline{W} = L^2(X,\mu)$ as $\mu$ is a Radon measure. So as before, we only need to prove the theorem in the case of $(W, \pi|_{W})$. For each $f \in W$ supported on one $H-$orbit \textcolor{black}{of a measurable set $U$}, that is, $$K_f \subseteq \bigsqcup_{h \in H} h.U,$$ where $K_f$ is the compact support of $f$ and the union is disjoint indexed by $H$, fix a point $p$ in $U$ and associate an element $A_f \in \ell^2(\mathbb{T})$, where $\mathbb{T}$ is the 2-tree based on $p$, via $$ A_f(h.p) = \left(\int_{h.U}|f|^2d\mu\right)^{\frac{1}{2}}.$$  Define $$K' = \left\{g \in H |~~ g ~~\text{or}~~ g^{-1} \in K\right\},$$ where $K$ is the same finite subset of H as in Lemma \ref{lemma:discrete}. For $f$, one has:
\begin{displaymath}
\begin{aligned}
&\int_{K_f}|\pi(g)f-f|^2d\mu \\
&\phantom{=\;\;}=\sum_{h \in H}\int_{h.U}|\pi(g)f-f|^2d\mu \\
&\phantom{=\;\;}\geq \sum_{h \in H}\left|\left(\int_{h.U}|\pi(g)f|^2d\mu\right)^{\frac{1}{2}}- \left(\int_{h.U}|f|^2d\mu\right)^{\frac{1}{2}}\right|^2 \\
&\phantom{=\;\;}= \sum_{h \in H}\left|A_{\pi(g)f}(h.p)-A_f(h.p)\right|^2\\
&\phantom{=\;\;}= \sum_{h \in H}\left|\left(\pi(g^{-1})A_{f}\right)(h.p)-A_f(h.p)\right|^2,
\end{aligned}
\end{displaymath}
where the second inequality is the triangle inequality. By Lemma \ref{lemma:discrete},
\begin{displaymath}
\begin{aligned}
&\max_{g \in K'}\|\pi(g)f-f\|^2 \\
&\phantom{=\;\;}\geq \max_{g \in K'}\sum_{h \in H}|(\pi(g)A_f)(h.p)-A_f(h.p)|^2 \\
&\phantom{=\;\;}=\max_{g \in K'}\|\pi(g)A_f-A_f\|^2 \\
&\phantom{=\;\;}\geq \eta \|A_f\|^2\\
&\phantom{=\;\;}=\epsilon' \|f\|^2,
\end{aligned}
\end{displaymath}
where $\epsilon'$ is a multiple of the constant $\eta$ in Lemma \ref{lemma:discrete}, as in this case we have $\sharp K' = 2 \sharp K$. If the compact set $K_f$ is not contained in one $H-$orbit, one can take an \textit{$H$-related cover} of $K_f$, then by the orthogonality similar to Fact 2 in Lemma \ref{lemma:discrete} and follow the last few lines in the proof of Lemma \ref{lemma:discrete}, one can also choose the pair $(K', \epsilon'')$, where $\epsilon''$ is a suitable multiple of $\epsilon'$, to complete the proof.
 \end{proof}
\begin{proof}[Proof of Theorem \ref{main:theorem}]
As any pseudo-Anosov subgroup acts freely on $\MF(S)$, by Lemma \ref{lemma:discrete} and Proposition \ref{proposition:properly discontinuous}, the theorem is true for $\cal{R} = \emptyset$. When $\cal{R} = S$ or $\cal{R}$ is of middle type, it is concluded by Lemma \ref{lemma:nondiscrete} and Proposition \ref{proposition:properly discontinuous}.

\end{proof}

\begin{remark}
 The same trick can be used to show that representations of mapping class groups in the space of $L^2-$functions on the Teichm\"{u}ller spaces with respect to Weil-Petersson volumes also have no almost invariant vectors. As one can show that such representations do not have non-trivial invariant vectors, we have the same conclusion about corresponding cohomology groups.
\end{remark}
\section{Classification of quasi-regular representations up to weak containment}\label{section:5}

\subsection{Irreducible decompositions}
\medskip
\noindent
As pointed out in Section \ref{section:des}, for unitary representations of mapping class groups associated to discrete measures on the space of measured foliations, both reducible and irreducible ones exist. By examining Example \ref{example:reducibility} carefully, one sees that, reducible representations have an irreducible decomposition. For any multi-curve $\gamma = \sum_{i = 1}^{k}c_i\gamma_i$ on $S$, where $c_i > 0$ for all $i$, we form $\tilde{\gamma} = \sum_{i = 1}^{k}\gamma_i$. Recall that $\{\gamma_i\}$ is a collection of pairwise disjoint isotopy classes of essential simple closed curves on $S$. As before, denote by \textcolor{black}{$G_{\gamma} = Mod(S,\gamma)$ and $G_{\tilde{\gamma}} = Mod(S,\tilde{\gamma})$} the corresponding subgroups of $\M(S)$. Hence $G_{\gamma}$ is a subgroup of $G_{\tilde{\gamma}}$ of finite index.
\begin{proposition}\label{proposition:irreducible}
Let $S = S_{g,n}$ be a compact surface with $3g + n \geq 4$ and $\gamma, \tilde{\gamma}$ as above.
\begin{description}
  \item[(1)] If the index of $G_{\gamma}$ in $G_{\tilde{\gamma}}$ is  one, then the associated representation in $\ell^2(\M(S)/G_{\gamma})$ of $\M(S)$ is irreducible.
  \item[(2)] If the index of $G_{\gamma}$ in $G_{\tilde{\gamma}}$ is $n > 1$, then the associated representation of $\M(S)$ in $\ell^2(\M(S)/G_{\gamma})$ is reducible.
\end{description}
\end{proposition}
\begin{proof}
(1) is obvious, since the representation $\ell^2(\M(S)/G_{\gamma})$ is $\ell^2(\M(S)/G_{\tilde{\gamma}})$ which is irreducible by Remark \ref{remark:multicurves}.\\
Now assume that $[G_{\tilde{\gamma}}:G_{\gamma}] = n > 1$. Let $X_{\gamma} = \M(S).\gamma$ and $Y_{\tilde{\gamma}} = \M(S).\tilde{\gamma}$, then $X_{\gamma}$ is a $\M(S)-$equivariant covering space of $Y_{\tilde{\gamma}}$ of degree $n$. So every $\ell^2-$function on $Y_{\tilde{\gamma}}$ defines an $\ell^2-$function on $X_{\gamma}$, and such correspondence produces a proper closed $\M(S)-$invariant subspace of $\ell^2(X_{\gamma})$, which implies the reducibility.
\end{proof}
\subsection{Classification up to weak containment}
\medskip
\noindent
We first fix some notations. Fix a hyperbolic structure on $S$. Denote by $\gamma = \sum_{i = 1}^{k}\gamma_i$ and $\delta = \sum_{i = 1}^{k}\delta_i$, that is, multi-curves on $S$ with coefficients all of $1$s. Such multi-curves will be called {\it $1-$multi-curves}. For any $1-$multi-curve $\gamma  = \sum_{i = 1}^{k}\gamma_i$ on $S$, we will call the union of geodesic representatives of $\gamma$ a {\it geometric multi-curve} and, for any $i$, the representative $\alpha_i$ a {\it geometric component}. Denote by $G_{\gamma}(G_{\delta})$ the corresponding subgroup of $\M(S)$, and by $\pi_{\gamma}(\pi_{\delta})$ the associated unitary representation on $\ell^2(\M(S)/G_{\gamma})(\ell^2(\M(S)/G_{\delta}))$. Let $\lambda_{S}$ be the regular representation of the mapping class group $\M(S)$ of $S$ on $\ell^2(\M(S))$. We first recall some definitions which can be found in \cite{Paris2002},\cite{BekkaHarpValette08}, \cite{BekkaKalantar}.\\

\noindent Let $G$ be a countable discrete group and H be a subgroup of $G$, the {\it commensurator} of H is defined to be $$Com_G(H) = \left\{g\in G: gHg^{-1} \bigcap H ~~\text{has finite index in }~~ H ~~\text{and}~~ gHg^{-1}. \right\}$$
A discrete group is said to be {\it C*-simple} if every unitary representation, which is weakly contained in the regular representation of $G$, is weakly equivalent to the regular representation. Let $\gamma$ and $\delta$ be geometric multi-curves, then $\gamma$ and $\delta$ are {\it of the same type} if there is an element $f$ in $\M(S)$ such that $f(\gamma) = \delta$. We say a subgroup $H$ of $G$ has the {\it spectral gap property} if the unitary representation $\ell^{2}(X)$ associated to the action $H \curvearrowright X = G/H - \{H\}$ does not have almost invariant vectors. In this section, we give a classification for unitary representations of $\M(S)$ associated to discrete measures.
\begin{lemma}\label{lemma:spectral}
Given a $1-$multi-curve $\gamma$ on $S$ and let $m$ be the number of its geometric components.
\begin{enumerate}
  \item If $m = 3g-3+n$, then $G_{\gamma}$ is amenable.
  \item If $1 \leq m < 3g-3+n$, then $G_{\gamma}$ has the spectral gap property.
\end{enumerate}
\end{lemma}
\begin{proof}
If $m = 3g-3+n$, then $G_{\gamma}$  is virtually abelian, thus it is amenable. For other cases, as $m < 3g-3+n$, one can cut $S$ along geometric components so that the resulting surface has at least one connected component that admits two pseudo-Anosov mapping classes generating a rank 2 pseudo-Anosov subgroup. Assume components admitting pseudo-Anosov mapping classes are labelled as $T_1,...,T_k$, two pseudo-Anosov mapping classes in each $\M(T_i)$ and the associated rank 2 pseudo-Anosov subgroup are also denoted by $\varphi_i,\psi_i, H_i$, respectively. Note that pseudo-Anosov homeomorphisms fix boundaries. Then define two maps $\varphi$ and $\psi$ on $S$ (thus their isotopy classes) by extending $\varphi = \prod_i\varphi_i$ and $\psi = \prod_i\psi_i$. Hence the subgroup $H$ generated by $\varphi$ and $\psi$ is a rank 2 free group. Moreover the action of $H$ on the set $X_{\gamma} - \{\gamma\}$ has trivial stabilizers. Otherwise, if an element $\phi$ in $H$ fix $\delta \in X_{\gamma} - \{\gamma\}$, then by the construction of $H$, the geometric intersection number of $\delta$ and $\gamma$ is nonzero and thus it intersects one of $T_i$. We cut $S$ along $\gamma$ so that $\delta$ becomes a family of isotopy classes of arcs. Since $\phi$ fixes $\delta$, up to some powers of $\phi$, it fixes each resulting isotopy class of arcs. But then it can be shown that, for some $i$, there is an element in $H_i$ that fixes the isotopy class of an essential simple closed curve, which contradicts the assumption that $H_i$ is a pseudo-Anosov subgroup. By Lemma \ref{lemma:discrete}, we can conclude that $G_{\gamma}$ has the spectral gap property.
\end{proof}

\begin{lemma}[Theorem A in \cite{BekkaKalantar}]\label{lemma:bekkaKa}
Let $G$ be a countable discrete group and $H$ be a subgroup of $G$ that has the spectral gap property. Let $L$ be a subgroup of $G$ satisfying $Com_G(L) = L$, then two unitary representations $\ell^2(G/H)$ and $\ell^2(G/L)$ of $G$ are weakly equivalent if and only if $L$ is conjugate to $H$.
\end{lemma}
\begin{theorem}\label{theorem:classification}
Let $S = S_{g,n}$ be a compact surface with $3g + n \geq 4$. Let $\gamma$ and $\delta$ be two $1-$multi-curves on  $S$ with $k,l$ geometric components, respectively.
\begin{description}
  \item[(1)] If at least one of $k, l$ is not $3g-3+n$, then the associated unitary representations $\pi_{\gamma}$ and $\pi_{\delta}$ are weakly equivalent if and only if $\gamma$ and $\delta$ are of the same type.
  \item[(2)] Suppose $S$ is not $S_{0,4}, S_{1,1}, S_{1,2}, S_{2,0}$. If the number of geometric components of $\gamma$ is $3g-3+n$, then $\pi_{\gamma}$ is weakly equivalent to the regular representation $\lambda_S$.
  \item[(3)] Suppose $S$ is not $S_{0,4}, S_{1,1}, S_{1,2}, S_{2,0}$. If the number of geometric components of $\gamma$ is not $3g-3+n$, then $\pi_{\gamma}$ is not weakly contained in $\lambda_S$.
\end{description}
\end{theorem}
\begin{proof}
For any $1-$multi-curve $\gamma$ on $S$, $Com_{\M(S)}(G_{\gamma}) = G_{\gamma}$ (see \cite{Paris2002}). Given two $1-$multi-curves $\gamma$ and $\delta$ with $k,l$ geometric components, respectively, such that at least one of $k$ and $l$ is not $3g-3+n$, then by Lemma \ref{lemma:spectral}, Lemma \ref{lemma:bekkaKa} and the fact that $G_{\gamma}$ is conjugate to $G_{\delta}$ if and only if $\gamma$ and $\delta$ are of the same type, we complete the proof for (1). For (2), by \cite{BridsonHarpe}, if $S$ is not $S_{0,4}, S_{1,1}, S_{1,2}, S_{2,0}$, the mapping class group $\M(S)$ is C*-simple. By the result of \cite{BKKO} which states that a discrete group is C*-simple if and only if, for any amenable subgroup $M$ of $G$, the quasi-regular representation $\ell^2(G/M)$ is weakly equivalent to the regular one. So combine with Lemma \ref{lemma:spectral}, we complete the proof of (2). The statement (3) is deduced from (2) and the definition of C*-simplicity.
\end{proof}
\begin{remark}
The ``only if" part of (1) is a stronger version of Corollary 5.5 in \cite{Paris2002}.
\end{remark}
\begin{remark}
If $S$ is one of $S_{0,4}, S_{1,1}, S_{1,2}, S_{2,0}$, it is easy to show that, if the number of components of $\gamma$ is $3g-3+n$, then $\pi_{\gamma}$ is weakly contained in the regular representation $\lambda_S$. However, for other types of $\gamma$, we don't know if $\pi_{\gamma}$ is weakly contained in $\lambda_S$. And we don't know what can be said about unitary representations corresponding to non-discrete measures on the space of measured foliations.
\end{remark}

\bibliography{unitary}

\bigskip
\noindent
Laboratoire J.A. Dieudonn\'e, CNRS, Universit\'e C\^ote d'Azur, 06108 NICE, France \\
\noindent \textbf{Email address: Biao.MA@univ-cotedazur.fr}
\end{document}